\documentclass[a4paper,12pt,reqno]{amsart}

 \usepackage[T2A]{fontenc}
 \usepackage[utf8]{inputenc}
 \usepackage[ukrainian,english]{babel}
 \usepackage{amssymb,upref}
\usepackage{amsthm}

\usepackage{graphicx}
\usepackage{mathrsfs}

\usepackage[text={130mm,200mm}]{geometry}

\theoremstyle{plain}
\newtheorem{theorem}{Theorem}
\newtheorem*{theorem*}{Theorem}

\newtheorem*{corollary*}{Corollary}
\newtheorem{lemma}{Lemma}
\newtheorem*{lemma*}{Lemma}

\newtheorem*{proposition*}{Proposition}

\newtheorem*{conjecture*}{Conjecture}
\theoremstyle{definition}
\newtheorem{definition}{Definition}
\newtheorem*{definition*}{Definition}
\theoremstyle{remark}

\newtheorem*{remark*}{Remark}

\begin{document}

\title[Asymptotic Mean of Digits of the $Q_s$--representation]{Asymptotic Mean of Digits of the $Q_s$--representation of the Fractional Part of a Real Number and Related Problems of Fractal Geometry and Fractal Analysis}

\author{M. V. Pratsiovytyi}
\address[M. V. Pratsiovytyi]{Institute of Mathematics of NAS of Ukraine,  Dragomanov Ukrainian State University, Kyiv, Ukraine\\
ORCID 0000-0001-6130-9413}
\email{prats4444@gmail.com}
\author{S. O. Klymchuk}
\address[S. O. Klymchuk]{Institute of Mathematics of NAS of Ukrain, Kyiv, Ukraine\\
ORCID 0009-0005-3979-4543}
\email{svetaklymchuk@imath.kiev.ua}


% Key words

\thanks{Scientific Journal of M. P. Drahomanov National Pedagogical University. Series 1. Physical and Mathematical Sciences. -- Kyiv: M. P. Drahomanov National Pedagogical University, 2011, No. 12, pp. 186–195.}

\begin{abstract}
We introduce a concept of asymptotic mean of digits (symbols) in the
$Q_s$--representation of a real number, that is a generalization of the
$s$--adic representation and have a self-similar geometry. We
discuss its relationship with the frequencies of digits and formulate
problems related to the concept. We study the topological,
metric, and fractal properties of the set of real numbers that have no asymptotic mean of $Q_s$--symbols. Also we study topological, metric and fractal
properties of the sets of real numbers that have asymptotic mean of
$Q_3$--symbols which is equal to value of digit frequency of number.
\end{abstract}

\maketitle
\section{Introduction}

Traditionally, a numeral system is understood as a collection of methods and techniques for the representation and encoding of numbers (natural, integer, real, complex, hypercomplex, etc.). Over the past decades the understanding of numeral systems has been significantly expanded and the role and significance of the base of a numeral system, its alphabet, and related structural components have been reconsidered.

Today, nontraditional numeral systems are widely used in mathematics and its applications. In particular, these include representations of numbers by nega--$s$--adic expansions, continued fractions (in particular $A_2$--fractions \cite{Xin}), expansions of numbers into Cantor series \cite{CantSer}, Lüroth series \cite{Lurot}, Engel series \cite{Engel}, Ostrogradsky--Sierpiński--Pierce series \cite{Sierp}, the mediant representation, as well as $Q^*$--representation and $Q_{\infty}$--representation \cite{PrQzobr}, etc. Each method of numbers representing corresponds to its own metric relations, its own geometry, topology, and metric and probabilistic theories. This creates the background for the development of the metric, probabilistic, and fractal theories of real numbers.

Every real number $x$ one can represent as the sum of its integer part $[x]$ and fractional part $\{x\}$, that is, for any real number $x$ the equality $x = [x] + \{x\}$ holds. We study objects of continuous mathematics and the fractional part of a real number is of particular importance to us since it allows us to address a number of questions in the metric and fractal theories of real numbers.

Let $\mathcal{A}_s=\{0,1,\ldots,s-1\}$ be the alphabet of the $s$--adic numeral system, and let $Q_s=\{q_0, \ldots, q_{s-1}\}$ be a fixed set satisfying the following conditions:
$$
\begin{cases}
1)\:\:\: q_i>0, \, \forall i\in \mathcal{A}_s;\\
2)\:\:\: q_0+\ldots+q_{s-1}=1.
\end{cases}
$$
Let $\beta_0=0$, $\beta_1=q_0$, and $\beta_j=q_0+q_1+\ldots+q_{j-1}$, where $j\in\mathcal{A}_s$.

It is known \cite{Pr} that \emph{for every number $x$ from the interval $[0;1]$ there exists a sequence $(\alpha_n)$, where $\alpha_n\in\mathcal{A}_s$, such that}
$$
x=\beta_{\alpha_1}+\sum\limits^{\infty}_{k=2}
  \left[ \beta_{\alpha_k}\prod\limits^{k-1}_{j=1}q_{\alpha_j} \right]. \eqno(1)
$$
The representation of a real number $x$ by series (1) is called its \emph{$Q_s$--representation}. We symbolically denote it as
$$
x=\Delta^{Q_s}_{\alpha_1\alpha_2\ldots\alpha_k\ldots}. \eqno(2)
$$
The notation is called the \emph{$Q_s$-representation (expansion) of the number $x$.} In this case  $\alpha_k(x)$ is called the $k$--th $Q_s$--digit of the number $x$.

For $q_i=\dfrac{1}{s}$ for all $i\in\mathcal{A}_s$ the $Q_s$--representation is equal to the classical \emph{$s$--adic representation.}

There exist at most two formally distinct $Q_s$--representations for any number $x$. Numbers that have two representations (one with period $(0)$ and the other with period $(s-1)$) are called \emph{$Q_s$--rational.} The remaining numbers have only one representation and are called \emph{$Q_s$--irrational.} The set of $Q_s$--rational numbers is countable.

To eliminate any ambiguity in defining the $k$--th $Q_s$--digit of a number, we consider only representations with period $(0)$.

After studying its geometry and developing the metric and probabilistic theories of numbers the $Q_s$--representation of real numbers has been used for constructing and investigating fractal linear sets and other objects. This is done by imposing conditions on the use of symbols from the alphabet of the $Q_s$--representation of numbers.

The present work is devoted to the introduction of a new concept --- asymptotic mean of digits of a number with respect to its $Q_s$--representation and to the function associated with it. In terms of this concept, mathematical objects with fractal properties can be formally and simply defined. We study the simplest properties asymptotic mean of digits function and indicate some of its applications.

\section{The object of our study and the related problems.}
\begin{definition}
If the limit $\lim\limits_{n\to\infty}\frac{1}{n}\sum\limits^{n}_{i=1}\alpha_i(x)=r(x)$, 
where $\alpha_i(x)$ is the $i$--th digit of the $Q_s$--representation of the number $x$ exists, then its value $r(x)$ is called the \emph{asymptotic mean} (or simply the \emph{mean}) of the digits of the number $x$.
\end{definition}

A number with a periodic $Q_s$--representation has a well-defined asymptotic mean of its digits. For example, $r\left(\Delta^{Q_s}_{(i)}\right)=i$, $r\left(\Delta^{Q_s}_{(ij)}\right)=\dfrac{i+j}{2}$, where $i,j\in\mathcal{A}_s$.

It follows from the definition that the function $r$ takes values in the interval $[0, s-1]$. We note that the value of $r(x)$ does not depend on any finite quantity of $Q_s$--symbols of $x$.

Let us provide an example of a number that does not possess an asymptotic mean of its digits:
$$
x_0=\Delta^{Q_s}_{dcddcc\ldots\underbrace{d\ldots d}_{2^{k-1}}
        \underbrace{c\ldots c}_{2^{k-1}}\ldots}, \,\,\,\,
        \text{where}\,\,\,\, c,d\in \mathcal{A}_s.
$$

Thus, the asymptotic mean function $r$ is defined on an everywhere dense subset of the interval $[0,1]$. On another everywhere dense subset of the interval $[0,1]$ it is undefined.

So we highlight problems formulated in terms of the asymptotic mean of digits of a number’s $Q_s$--representation, namely:

\begin{enumerate}
  \item to study topological, metric, fractal and other properties of the sets of real numbers for which the asymptotic mean of digits exists and satisfies certain conditions, or does not exist;
  \item to study the properties of the function $r$, in particular the set of its values, level sets, and the fractal properties of its graph;
  \item to study the properties of linear combinations of level sets;
  \item to determine  solutions of the equation $r(x) = kx + b$, in particular its invariant points;
  \item to study the fractal properties of the distribution function $Y = r(X)$, where $X$ is a random variable with a given (e.g., uniform) distribution;
  \item to find the normal properties of numbers in terms of the asymptotic mean of digits;
  \item to investigate the content of the sets
       $$S_{\theta}\equiv \{x: r(x)=\theta\},$$ where $\theta\in[0;s-1]$, within Cantor--type sets
        $$C[Q_s,V]\equiv\left\{x:x=\Delta^{Q_s}_{\alpha_1\alpha_2\ldots\alpha_k\ldots}, \alpha_i\in V, i\in\mathcal{A}_s\right\},$$
        i.e., the intersection $C[Q_s, V] \cap S_{\theta}$.
\end{enumerate}

\section{The relationship between the asymptotic mean of digits and their frequencies. The asymptotic mean of digits function}

Let $N_i(x,k)$ be the quantity of the digits $i \in \mathcal{A}_s$ in the $Q_s$--representation 
$\Delta^{Q_s}_{\alpha_1\alpha_2\ldots\alpha_k\ldots}$ of a number $x \in [0,1]$ up to and including the $k$--th position, that is,
$$
N_i(x,k)=\# \{j:\, \alpha_j (x)=i, \, j\leqslant k\}.
$$

The \emph{frequency (asymptotic frequency) of the symbol $i$} in the $Q_s$--representation of the number $x$ is defined as the limit (if it exists)
$$
\nu_i(x)=\lim\limits_{k\to\infty}\displaystyle\frac{N_i(x,k)}{k}.
$$

We note that any finite set of digits of a number $x$ does not affect the frequency of its digits.

It is clear that the frequency function $\nu_i$ of the digit $i$ in the $Q_s$--representation of a number $x$ is well-defined for $Q_s$--irrational numbers. For $Q_s$--rational numbers it is properly defined once we adopt the convention of using only representations with period $(0)$.

As usual, if a number has period $(c_1 c_2 \ldots c_m)$, with $c_i \in \mathcal{A}_s$, then each symbol has a well--defined frequency \cite{Pr}.

The asymptotic mean of digits of a number one can regard as a certain analogue of the concept of digit frequency. In the binary numeral system the asymptotic mean of digits of a number coincides with the frequency of the digit $1$, that is, $r(x) = \nu_1(x)$. Indeed, from the definitions of digit frequency and digits asymptotic mean it follows that
$$
r(x)=\lim\limits_{n\to\infty}\frac{\alpha_1(x)+\ldots+\alpha_n(x)}{n}=
     \lim\limits_{n\to\infty}\frac{0 N_0(x,n)+1 N_1(x,n)}{n}=
     \lim\limits_{n\to\infty}\frac{N_1(x,n)}{n}=\nu_1(x).
$$

It follows from the above  that if the frequencies of all digits exist, then the asymptotic mean of digits also exists, and the following relation holds:
$$
r(x)=\nu_1(x)+2\nu_2(x)+\ldots+(s-1)\nu_{s-1}(x).
$$

Thus, for Besicovitch--Eggleston sets
$$
E[\tau_0,\tau_1,\ldots,\tau_{s-1}]=\{x:\nu_i(x)=\tau_i,\,\,i\in\mathcal{A}_s\},
$$
where $0<\tau_i<1$, $\tau_0+\tau_1+\ldots+\tau_{s-1}=1$ the inclusion $E[\tau_0,\tau_1,\ldots,\tau_{s-1}]\subset S_{\theta}$ holds if $\tau_1+2\tau_2+\ldots+(s-1)\tau_{s-1}=\theta$.

The Hausdorff--Besicovitch fractal dimension of such sets is calculated \cite{Pr} using the formula
$$
\alpha_0(E)=\frac{\ln\tau_0^{\tau_0}\tau_1^{\tau_1}\ldots\tau_{s-1}^{\tau_{s-1}}}
                 {\ln q_0^{\tau_0}q_1^{\tau_1}\ldots q_{s-1}^{\tau_{s-1}}}.
$$

\section{Normality of numbers and the role of normality in the theory of singular functions}

We can express the ``normal'' properties of a real number in terms of the concept of digit frequency in its representation. We call a property of an element {\it normal} if it holds for the vast majority (in a certain sense) of elements in that set.

A number $x$ is called \emph{normal on the base $s$ (weakly normal)} if for each $i \in \mathcal{A}_s$ the frequency exists and satisfies $\nu_i(x) = s^{-1}$. A number $x$ that is normal in every natural base $s \ge 2$ is called \emph{normal}.

An example of a number that is normal in base $s$ is 
$\Delta^s_{(0\,1\,2\,\ldots\,[s-2]\,[s-1])}$.

The concept of a normal numbers was introduced by É. Borel in 1909. Borel’s classical theorem \cite{Bor1} asserts that \emph{The set of normal numbers (the set of numbers from the interval $[0,1]$ whose $s$--adic representations have all digit frequencies equal to $s^{-1}$) has full Lebesgue measure.} Thus, almost all real numbers are normal.

D. Champernowne \cite{Champ} constructed the first interesting examples of non-periodic numbers that are normal in some base $s$ in 1933. He showed that the number
$$
\Delta^{10}_{1\, 2\, 3\, 4\, 5\, 6\, 7\, 8\, 9\, 10\, 11\, 12\, 13\,
14\, 15\, 16\, 17\, \ldots},
$$
whose digits consist of all natural numbers written in decimal notation, is weakly normal in base 10. He also conjectured that the number
$$
\Delta^{10}_{1\, 2\, 3\, 5\, 7\, 11\, 13\, 17\, 19\, 23\, \ldots},
$$
whose decimal expansion lists all prime numbers in increasing order, is normal in base 10. Copeland and Erdős proved this conjecture in 1946 \cite{CopErd}.

Borel’s theorem implies that the set of non-normal numbers (numbers that are not normal) has zero Lebesgue measure.

Let us consider the set of numbers in the interval $[0;1]$ whose asymptotic mean of digits does not exist, that is, the set
$$
S=\left\{x:r(x) \,\,\,\text{does not exist}\,\, \right\}.
$$

\begin{theorem}
The set $S$ is a continuous, everywhere dense and everywhere discontinuous set of zero Lebesgue measure. The set $S$ is superfractal, that is, its Hausdorff–Besicovitch dimension $\alpha_0(S)$ equals $1$.
\end{theorem}

\begin{proof}
\emph{ Continuity. } 
Let us consider the following form of the representation of a real number $x \in S$:
$$x\equiv
\Delta^{Q_s}_{\underbrace{0\ldots0}_{a_{01}}
          \underbrace{1\ldots1}_{a_{02}}\ldots
          \underbrace{[s-1]\ldots[s-1]}_{a_{(s-1)1}}
          \underbrace{0\ldots0}_{a_{02}}
          \underbrace{1\ldots1}_{a_{12}}\ldots
          \underbrace{[s-1]\ldots[s-1]}_{a_{(s-1)2}}\ldots}
\equiv\overline{\Delta}^{Q_s}_{\alpha_1\alpha_2\ldots\alpha_n\ldots}.$$

To a number $[0;1]\ni x_0=\Delta^{Q_s}_{b_1b_2\ldots b_n\ldots}$, where $b_i
\in\{0,s-1\}$   we associate the number $\hat{x}=f(x)=x_0*x$,
where $f(x)=\Delta^{Q_s}_{\beta_1\beta_2\ldots\beta_n\ldots}$ and
$$\beta_{s^n+1}(\hat{x})=b_n,\:\: \beta_j(\hat{x})=\alpha_j,\:\: \text{if } j\neq s^n+1,\:\: n=1,2,\ldots .$$

Distinct pairs $x_0$ and $x'_0$ correspond to distinct numbers $f(x_0)$ and $f(x'_0)$. Since we can choose $x_0$ in a continuum of ways then the resulting set of numbers $\hat{x}$ is also continuous. In this way we construct a mapping from the set $S$ onto a certain continuous set $S'$. Hence, the set $S$ is continuous. 

\emph{ Everywhere density.} The set $S$ is everywhere dense set since for any cylinder interval $\Delta^{Q_s}_{c_1c_2\ldots c_k}$ we can find a point  $S\ni x_0=\Delta^{Q_s}_{c_1c_2\ldots c_k\alpha_1\alpha_2\ldots\alpha_k\ldots}$ that belongs to this cylinder interval.

\emph{Everywhere discontinuity.} Let $x_1, x_2 \in S$ with $x_1 > x_2$. Starting from some index $(k+1)$, the $Q_s$--representations of $x_1$ and $x_2$ will differ; that is, there exists a natural number $k$ such that $c_{k+1}(x_1)\neq c_{k+1}(x_2)$  while $c_i(x_1) = c_i(x_2)$ for $i \le k$. In other words, 
$x_1=\Delta^{Q_s}_{c_1\ldots c_k\alpha_1\alpha_2\ldots}$ and $x_2=\Delta^{Q_s}_{c_1\ldots c_k\beta_1\beta_2\ldots}$.  
However, between $x_1$ and $x_2$ we can always find a rational point 
$\Delta^{Q_s}_{c_1(x_2) \ldots c_k(x_2) c_{k+1}(x_2) (0)} = x_0 \not\in S$.  
Therefore, the set $S$ is everywhere discontinuous.

\emph{Lebesgue Measure.} The set $S$ consists of all numbers whose limit
$\lim\limits_{n \to \infty} \frac{1}{n} \sum\limits_{i=1}^{n} \alpha_i(x)$
does not exist, in other words the limit
$$
\lim_{n\to\infty}\left(\dfrac{N_1(x,n)}{n}+
\dfrac{2N_2(x,n)}{n}+\ldots+\dfrac{(s-1)N_{s-1}(x,n)}{n}\right).
$$
does not exist.  

This means that at least one of the limits
$\lim\limits_{n \to \infty} \frac{N_1(x,n)}{n},$
$\lim\limits_{n \to \infty} \frac{N_2(x,n)}{n}, \ldots,$
$\lim\limits_{n \to \infty} \frac{N_{s-1}(x,n)}{n}$ 
does not exist. Thus, at least one of the $Q_s$--digits frequencies 
$\nu_1, \nu_2, \ldots, \nu_{s-1}$ is undefined. Therefore, the set $S$ contains only abnormal numbers.

The Lebesgue measure of the set of abnormal numbers equals zero. Since $S$ is a subset of the set of abnormal numbers, we have $\lambda(S) = 0$.

\emph{Superfractality.}  
Let us consider the set
$$
A_k=\left \{x:
         x=\Delta^{Q_s} {[\alpha_1 ... \alpha_{2k}01]
                     [\alpha_{t_1+1} ... \alpha_{t_1+2^2}0011]...
                     [\alpha_{t_{m-1}+1} ... \alpha_{t_{m-1}+2^m}
                     \underbrace{0...0}_{2^{m-1}}\underbrace{1...1}_{2^{m-1}}]}
    \right\},
$$
where $k$ is an arbitrary fixed natural number,  $t_m=2(k+1)(2^m-1)$,
$m\in N$ and $\alpha_{t_m+j}\in \mathcal{A}_s$,
$j=\overline{1,2^{m+1}k}$. We know \cite{PrTorb} that
$$\alpha_0\left(\bigcup_{k\in N}A_k\right)=\sup \alpha_0(A_k)=\sup_{k\in N}\frac{k}{k+1}=1.$$

It is easy to see that  $S\supset\bigcup\limits_{k\in N}A_k$.
Indeed \cite{PrTorb}, a number $x_0 \in A_k$ does not have a frequency for the digit $0$, since proceeding to infinity along different sequences $t_m$ and $t'_m$ yields different limits 
$\lim\limits_{m\to\infty}\dfrac{N_0(x_0,t_m)}{t_m}$ and $\lim\limits_{m\to\infty}\dfrac{N_0(x_0,t'_m)}{t'_m}$ 
(and similarly for the frequency of the digit $1$). Therefore, the asymptotic mean of digits for $x_0 \in A_k$ does not exist, which implies $x_0 \in S$.

On the other hand, the set $S$ is a subset of the abnormal numbers $A$, that is, $S \subset A$.  
According to the properties of the Hausdorff--Besicovitch dimension, if 
$\bigcup\limits_{k\in N}A_k\subset S\subset A$ then 
$\alpha_0(\bigcup\limits_{k\in N}A_k)\leqslant\alpha_0(S)\leqslant\alpha_0(A)$, 
that is, $1\leqslant\alpha_0(S)\leqslant1$. 
Hence, $\alpha_0(S) = 1$, which implies that the set $S$ is superfractal.
\end{proof}

Now we consider the set of numbers in the interval $[0,1]$ whose asymptotic mean of digits exists and equals the frequency of the digit $i$, where $i \in \{0,1,2\}$. That is, define the set $M=\underset{i}{\bigcup} M_i$, where $$M_i=\{x: r(x)=\nu_i(x)\}.$$

Let us consider the set $M_0$. According to the definition of the asymptotic mean of digits we have $$r(x)= \nu_1(x)+2\nu_2(x),$$ thus $\nu_1(x)+2\nu_2(x)= \nu_0(x)$. On the other hand $\nu_0(x)+\nu_1(x)+\nu_2(x)=1$. Hence, it follows that $$2\nu_1(x)+3\nu_2(x)=1.$$

\begin{lemma}
The set $M_0$ is a continuous, everywhere dense set of zero Lebesgue measure, whose Hausdorff–Besicovitch dimension is at least $0.87$.
\end{lemma}

\begin{proof}
\emph{Continuity.} Since every point $x$ in the set $M_0$ satisfies the relation $2\nu_1(x)+3\nu_2(x)=1,$ the set $M_0$ contains all Besicovitch--Eggleston subsets  $E\equiv E[\tau_0,\tau_1,\tau_2]$ with $$2\tau_1+3\tau_2=1.$$ Therefore, the continuity of $M_0$ follows from the continuity of the sets $E$.

\emph{Everywhere density.} The set $M_0$ is everywhere dense, since every cylinder  $\Delta^3_{c_1\ldots c_k}=\{x:\alpha_j(x)=c_j,\:j=\overline{1,k}\}$ contains points from $M_0$. Specifically, if $x_0=\Delta^3_{\alpha_1\alpha_2\ldots\alpha_n\ldots}\in M_0$ then the point $x=\Delta^3_{c_1\ldots c_k\alpha_1\ldots\alpha_n\ldots}\in M_0 \cap\Delta^3_{c_1\ldots c_k}$.

\emph{Lebesgue Measure.} Since the set $M_0$ contains all Besicovitch--Eggleston subsets satisfying 
$$2\tau_1+3\tau_2=1,$$
it obviously does not contain the set of normal numbers $E\left[\dfrac{1}{3},\dfrac{1}{3},\dfrac{1}{3}\right]$. As the Lebesgue measure of the set of non--normal numbers equals zero, we have $\lambda(M_0)=0$.

\emph{Hausdorff--Besicovitch dimension.}  From the conditions $\nu_0(x)+\nu_1(x)+\nu_2(x)=1$ and $\nu_0=\nu_1(x)+2\nu_2(x)$ we determine
$\nu_1=\dfrac{1}{2}-\dfrac{3}{2}\nu_2$ and $\nu_0=\dfrac{1}{2}-\dfrac{3}{2}\nu_2+2\nu_2=\dfrac{1}{2}+\dfrac{1}{2}\nu_2.$

To compute the Hausdorff--Besicovitch dimension, we study the maximum of the function
$$
\begin{array}{ll}
y&=-\displaystyle\frac{\ln  x^{x}\cdot
\left(\frac{1}{2}-\frac{3}{2}x\right)^
{\frac{1}{2}-\frac{3}{2}x}\cdot
\left(\frac{1}{2}+\frac{1}{2}x\right)^{\frac{1}{2}+\frac{1}{2}x}}{\ln
3},
\end{array}
$$ where $x\in[0;1]$, i.e.\\
$
\begin{array}{ll}
y&=\displaystyle\frac{\ln x^{x}+\ln
\left(\frac{1}{2}-\frac{3}{2}x\right)^{\frac{1}{2}-\frac{3}{2}x}+
\left(\frac{1}{2}+\frac{1}{2}x\right)^{\frac{1}{2}+\frac{1}{2}x}}{-\ln
3},
\end{array}
$\\
$
\begin{array}{ll}
y'&=-\displaystyle\frac{1}{\ln3}\left(\ln x+1-
 \frac{3}{2}\ln\left(\frac{1}{2}-\frac{3}{2}x\right)-\frac{3}{2}
+\frac{1}{2}\ln\left(\frac{1}{2}+\frac{1}{2}x\right)+\frac{1}{2}\right)=\\
&=-\displaystyle\frac{1}{\ln 3} \left(\ln
x-\frac{3}{2}\ln\left(\frac{1}{2}-\frac{3}{2}x\right)+
\frac{1}{2}\ln\left(\frac{1}{2}+\frac{1}{2}x\right)\right)=\\
&=-\displaystyle\frac{1}{2\ln 3}\left(\ln\displaystyle\frac{x^2
\left(\frac{1}{2}+\frac{1}{2}x\right)}{\left(\frac{1}{2}-\frac{3}{2}x\right)^3}\right),
\end{array}
$\\
$
\begin{array}{ll}
y'& =0,\:\:\:\displaystyle\frac{x^2
\left(\frac{1}{2}+\frac{1}{2}x\right)}{\left(\frac{1}{2}-\frac{3}{2}x\right)^3}=1,
\end{array}
$\\
\[
\begin{cases}
4x^2(1+x)-(1-3x)^3=0,\\
x\neq\frac{1}{3};\\
\end{cases}
\]
$$31x^3-23x^2+9x-1=0.$$
Using Cardano’s formulas we obtain
$$x=\frac{\sqrt[3]{-15374-2\sqrt{66394497}}-\sqrt[3]{-15374+2\sqrt{66394497}}}{93}.$$
Thus, $x\approx 0,1655$. Then $\nu_0\approx 0,5828$, $\nu_1\approx
0,2517$, $\nu_2\approx 0,1655$.
$$y_{max}=-\frac{1}{\ln3}\ln0,5828^{0,5828}0,2517^{0,2517}0,1655^{0,1655}\approx 0,8733.$$
\end{proof}

\begin{lemma}
The set $M_1$ is a continuous, everywhere dense set of zero Lebesgue measure, whose Hausdorff--Besicovitch dimension is at least $\log_2 3$.
\end{lemma}

\begin{proof}
We establish the continuity, everywhere density, and Lebesgue measure of the set $M_1$ in a similar way.

\emph{Hausdorff--Besicovitch dimension.} From the conditions $\nu_0(x)+\nu_1(x)+\nu_2(x)=1$ and $\nu_2(x)=0$ we find $\nu_1(x)=1-\nu_0(x).$

To compute the Hausdorff--Besicovitch dimension, we study the maximum of the function
$$y=-\frac{\ln\nu_0(x)^{\nu_0(x)}(1-\nu_0(x))^{1-\nu_0(x)}}{\ln 3}.$$

The function attains its maximum at
$\nu_0(x)=\nu_1(x)=\dfrac{1}{2}$. Thus,
$$y_{max}=-\frac{\ln\frac{1}{2}^{\frac{1}{2}}\frac{1}{2}^{\frac{1}{2}}}{\ln3}=\log_2 3.$$
\end{proof}

Let us consider the subset $M_2=\{x: r(x)=\nu_2(x)\}.$ By definition we have
$$r(x)= \nu_1(x)+2\nu_2(x).$$
Hence $\nu_1(x)+2\nu_2(x)=\nu_2(x)$ which gives $\nu_1(x)+\nu_2(x)=0$, so that
$\nu_1(x)=\nu_2(x)=0$, $\nu_0(x)=1.$

\begin{lemma}
The set $M_2$ is anomalously fractal and everywhere dense set.
\end{lemma}
\begin{proof}
\emph{Everywhere density.} The set $M_2$ is everywhere dense set since every cylinder $\Delta^3_{c_1\ldots c_k}=\{x:\alpha_j(x)=c_j,\:j=\overline{1,k}\}$ contains points from $M_2$. Specifically, if  $x_0=\Delta^3_{\alpha_1\alpha_2\ldots\alpha_n\ldots}\in M_1$ then the point
$x=\Delta^3_{c_1\ldots c_k\alpha_1\ldots\alpha_n\ldots}\in M_1 \cap\Delta^3_{c_1\ldots c_k}$.

\emph{Hausdorff--Besicovitch dimension.} Clearly, $M_2\equiv E[1,0,0]$ so the Hausdorff--Besicovitch dimension of $M_2$ is $$\alpha_0(E[1,0,0])=\frac{\ln1^10^00^0}{-\ln 3}=0.$$ Hence, $M_2$ is an anomalously fractal set.
\end{proof}

\end{document}